\documentclass[12pt,a4paper,reqno]{amsart}

\usepackage[latin1]{inputenc}
\usepackage{graphicx}
\usepackage{amssymb, amsmath}
\usepackage{color}
\usepackage{geometry}
\usepackage{enumerate}
\usepackage{amsthm}
\usepackage{amsmath}
\usepackage{soul}
\usepackage{tikz}
\usepackage{newfloat}

\DeclareMathOperator{\Dom}{Dom}
\DeclareMathOperator{\re}{Re}

\DeclareFloatingEnvironment[fileext=dia,within=section]{diagram}

\usepackage[all]{xy}
\usepackage{tikz}
\usetikzlibrary{arrows,chains,matrix,positioning,scopes}
\makeatletter
\tikzset{join/.code=\tikzset{after node path={%
			\ifx\tikzchainprevious\pgfutil@empty\else(\tikzchainprevious)%
			edge[every join]#1(\tikzchaincurrent)\fi}}}
\makeatother
\tikzset{>=stealth',every on chain/.append style={join},
	every join/.style={->}}
\tikzstyle{labeled}=[execute at begin node=$\scriptstyle,
execute at end node=$]
\usepackage[colorlinks=true,linkcolor=blue,urlcolor=blue,citecolor=blue]{hyperref}
\newtheorem{theorem}{Theorem}[section]
\newtheorem{definition}[theorem]{Definition}

\newtheorem{corollary}[theorem]{Corollary}
\newtheorem{proposition}[theorem]{Proposition}

\newtheorem{example}[theorem]{Example}
\newtheorem{remark}[theorem]{Remark}

\newcommand{\primos}{\textit{\textbf {p}}}

\newcommand{\T}{\mathbb{T}}
\newcommand{\D}{\mathbb{D}}
\newcommand{\N}{\mathbb{N}}
\newcommand{\C}{\mathbb{C}}
\newcommand{\Ha}{\mathcal{H}}
\newcommand{\Hd}{\mathcal{H}^2}

\DeclareFontFamily{U}{wncy}{}
\DeclareFontShape{U}{wncy}{m}{n}{<->wncyr10}{}
\DeclareSymbolFont{mcy}{U}{wncy}{m}{n}
\DeclareMathSymbol{\Sh}{\mathord}{mcy}{"58}
\usepackage{amsmath,amsfonts,amssymb,amsthm}
\usepackage{mathtools}
\usepackage{tikz}
\usetikzlibrary{matrix}

\let\norm\undefined 
\DeclarePairedDelimiter\norm{\lVert}{\rVert}

\begin{document}
\title[Splitting the Riesz basis condition for systems of dilated functions]{Splitting the Riesz basis condition for systems of dilated functions through Dirichlet series}

\author{Jorge Antezana}
\address{Departamento de Matematica, Universidad Nacional de La Plata,
and  IAM-
CONICET, Argentina}
\email{antezana@mate.unlp.edu.ar}
\thanks{ Supported by CONICET-PIP 152,  ANPCyT PICT 2015-1505 and National University of La Plata Grant 11X829 (UNLP)
.}

\author{Daniel Carando }
\address{Departamento de Matem\'atica, Facultad de Cs. Exactas y Naturales,
		Universidad de Buenos Aires and IMAS-UBA-CONICET, Int.~G\"uiraldes s/n, 1428, Buenos Aires, Argentina}
\email{dcarando@dm.uba.ar}
\thanks{ Supported by CONICET-PIP 11220130100329CO and  ANPCyT PICT 2018-04104
.}
\author{
Melisa Scotti}
\address{Departamento de Matem\'atica, Facultad de Cs. Exactas y Naturales,
	Universidad de Buenos Aires and IMAS-UBA-CONICET, Int.~G\"uiraldes s/n, 1428, Buenos Aires, Argentina} \email{mscotti@dm.uba.ar}
\thanks{ Supported by a CONICET doctoral fellowship,  CONICET-PIP 11220130100329CO and  ANPCyT PICT 2018-04104
.}

\begin{abstract}
	 Inspired by the work of Hedenmalm, Lindqvist and Seip, we consider different properties of dilations systems of a fixed function $\varphi \in L^2(0,1)$.  More precisely, we study when the system $\{\varphi(nx)\}_n$ is a Bessel sequence, a Riesz sequence, or it satisfies the lower frame bound. We are able to characterize these properties in terms of multipliers of the Hardy space $\Ha^2$ of Dirichtet series and, also, in terms of Hardy spaces on the infinite polytorus. We also address the multivariate case.
\end{abstract}

\maketitle

\thispagestyle{empty}

\section{Introduction}

Take a function $\varphi \in L^2(0,1)$  and consider  it as a 2-periodic odd function defined  on $\mathbb R$. Hedenmalm, Lindqvist and Seip  in \cite{HeLiSe} studied the system of dilations $\varphi(x), \varphi(2x), \varphi(3x),\dots$ in order to characterize the functions $\varphi$ for which this system is a Riesz basis or is a complete system in  $L^2(0,1)$.  They solved the Riesz basis problem in terms of multipliers on the Hardy space $\Hd$ of Dirichlet series (see Theorem~\ref{teoseip} below). They also characterized those multipliers as Dirichlet series that  define analytic functions on the open right half plane which are bounded and bounded away from zero (see Theorem~\ref{Multipliers}). Those results, connecting  the Riesz basis (and completeness) problem for the system $\{\varphi(n\cdot):n\in \mathbb N \}$ to Dirichlet series and also to power series in infinitely many variables, follow some ideas of Bohr and Beurling (see \cite{BEU89,Bohr1913b,Bohr1913}). The fundamental work \cite{HeLiSe} marked a new starting point to important advances in different branches of the theory of Dirichlet series, infinite dimensional holomorphy, functional and Fourier analysis, many of which are documented in the recent monographs \cite{DirSer} and \cite{QueQue}.

In this work we consider different variants of the Riesz basis problem studied in \cite{HeLiSe}. Firstly, since the dilation system is a frame if and only if it is a Riesz basis,  we splits Hedenmalm, Lindqvist and Seip's result for Riesz bases into two assertions. In Theorem ~\ref{teo1} we characterize those $\varphi$ for which the corresponding system of dilations $\{\varphi(nx):n\in \mathbb N \}$ forms a Bessel sequence or satisfies a lower frame bound. The first part of the theorem was already proved in \cite{HeLiSe}. However, when the lower frame bound is studied alone, we loose the boundedness of the involved operators and, then, the original ideas do not work. As mentioned in Remark \ref{ciclicos}, the lower frame bound implies the ciclycity in $\Hd$ of the associated Dirichlet series.

Secondly, we decouple the Riesz basis condition in a different way. A Riesz basis is a complete Riesz sequence. The completeness of the system of dilations $\{\varphi(nx):n\in \mathbb N \}$ was studied in \cite{HeLiSe}, and later in  \cite{DanGuo21,NI12}. However, to the best of our knowledge, the systems of dilation that form a Riesz sequence, not necessarily complete, has not been studied. In Section~\ref{Srieszseq} we fill this gap, and we characterize  those $\varphi$ for which the dilation system is a Riesz sequence or an orthonormal sequence. We also give applications to shift invariant spaces. All these properties of the dilation system are characterized in terms of properties of related Dirichlet series and also of related functions in the infinite polytorus.

In Section  \ref{SExamples} we present different examples in which our results are applied. We show dilation systems satisfying just one of the frame bounds. We also give nontrival examples of dilation systems that are orthonormal sequences.
Finally, in Section \ref{SMultivariateCase} we extend our results to the multivariate case, making a natural connection between functions on $(0,1)^k$ and  multiple Dirichlet series.

The definitions and general results that are used throughout the article are mainly given in Section~\ref{preliminar}. For further information on Dirichlet series and their connection with infinite dimensional holomorphy or Fourier analysis in the infinite polytorus we refer to the monograph \cite{DirSer}.

\section{Preliminaries}\label{preliminar}

Given a function $\varphi \in L^2(0,1)$, consider it as a $2-$periodic odd function on $\mathbb{R}$ and define
$$\varphi_n(x) = \varphi(nx), $$
for all  $n  \in \mathbb{N}$.  Among the systems of dilations, the only that constitutes an orthogonal basis is the generated by the functions  $\varphi(x)=
C\sin(\pi x)$ for $C \neq 0$ (see, for example, \cite{BourMen45}). In particular, we know that the system $\{\sqrt{2} \sin(\pi nx)\}_n$ is an orthonormal basis for the space $L^{2}(0,1)$. This motivates the following identity
\begin{equation*}
\varphi (x) = \sum_{n=1}^{\infty} a_n \sqrt{2} \sin(\pi n x).
\end{equation*}

From now on, $\{e_n\}_n$ denote the orthonormal basis given by  $e_n (x):=\sqrt{2} \sin(\pi n x)$. With this notation, the above identity reads as
\begin{equation}\label{escritura phi}
\varphi(x)= \sum_{n=1}^\infty a_n e_n(x).
\end{equation}

Beurling's  idea was to associate $\varphi$ with a Dirichlet series. To see this, we need to introduce the space of all Dirichlet series whose coefficients are square summable. This space is known as
\begin{equation*}
\Ha^2 =  \left\{ \sum_n a_n n^{-s} : \| \sum_n a_n n^{-s} \|_{\Ha^2} = \left(\sum_{n} a_n^2 \right)^{1/2} < +\infty \right\},
\end{equation*}
and is a Hilbert space with de inner product given by $\langle \sum_n a_n ^{-s }, \sum b_n n^{-s} \rangle = \sum_n a_n \overline{b_n}$.

We now define some operators of key importance  throughout this work. Let
$S: L^2(0,1) \rightarrow \mathcal{H}^2$  be the operator which associates to each $f= \sum_n c_n e_n(x)$ in $ L^2 (0,1)$ the Dirichlet series
\[
Sf(s) =  \sum_{n=1}^\infty c_n n^{-s}.
\]
It is clear that the operator $S$ is an isometry between $L^2 (0,1)$ and $\mathcal{H}^2$ since it maps $\{e_n \}_n$ to $\{n^{-s}\}_n$ both orthonormal bases. Additionally, let $\mathcal{F}$ be the set of finite  linear combinations of ${e_n}$.  We can associate to each   $f= \sum_n c_n e_n \in \mathcal{F} $  the function
\[
T_{\varphi} f(x) = \sum_n c_n \varphi_n (x)= \sum_{n} \langle f, e_n \rangle \, \varphi_n (x).
\]
This defines an operator  $T_{\varphi}: \mathcal{F} \rightarrow L^ 2(0,1)$  which satisfies the following important identity
\begin{equation}\label{identidad}
S(T_{\varphi}f)(s) = S\varphi(s) Sf(s), \qquad  Re(s)> \frac{1}{2}.
\end{equation}
This means that the operator $T_{\varphi}$  corresponds to multiplying by $S\varphi$ in the Dirichlet series setting.

To see this, we start by computing $T_{\varphi}$ in basis $\{e_n\}$. First, note that the basis $\{e_n\}$, that has been fixed, is also given by dilations of the function  $\sqrt{2}\sin(\pi x)$ so we have that
\[
\varphi(nx) = \sum_k a_k e_k(nx) = \sum_k a_k e_{kn} (x).
\]
From now on let $f= \sum_n c_n e_n$, then
\begin{align}
T_{\varphi} \left(\sum_{n} c_n e_n \right) = \sum_n c_n \varphi_n = \sum_n c_n \left(\sum_m  a_m  e_{nm}\right).
\end{align}
Then from above follows that
\begin{equation}\label{multiD1}
T_{\varphi} (f) = \sum_n \left(\sum_{dm=n} c_d a_m \right) e_n.
\end{equation}
On the other hand, the space of all (formal) Dirichlet series is denoted by $\mathfrak{D}$. This space forms an algebra with the so-called `Dirichlet multiplication'  given by
\begin{equation}\label{multiD2}
S\varphi * Sf = \left( \sum a_n  n^{-s}\right) * \left( \sum c_n n^{-s} \right) = \sum_{n} \left( \sum_{dm= n} c_d a_m \right)n^{-s}.
\end{equation}
It is well known that the invertible elements are exactly those series $\sum a_n n^{-s}$ such that $a_1 \neq 0$.

Combining \eqref{multiD1} and \eqref{multiD2} we obtain the formal relation
\[
	S \circ T_{\varphi} = M_{S\varphi}\circ S,
\]
where $M_{S\varphi}$ is the multiplication operator from  $\mathfrak{D}$ to $ \mathfrak{D}$ given by $M_{S\varphi}(E)= S{\varphi}*E$.
Then, since all the Dirichlet series involved belong to $\mathcal{H}^2$ we deduce the pointwise identity \eqref{identidad} for $ Re(s) > 1/2$.

Also, a Dirichlet series $D$ is a multiplier for $\Ha^2$ if $D*E$ is in $\Ha^2$ for all $E \in \Ha^2$. Following the notation of \cite{HeLiSe} we write $\mathcal{M}$ for the set of all multipliers of $\Ha^2$. Each Dirichlet series defines a multiplication operator $M_D:  \mathfrak{D} \rightarrow \mathfrak{D}$. It is clear that $D$ belongs to $\mathcal{M}$ if and only if the multiplication operator $M_D$ defines a bounded operator from $\Ha^2$ to $\Ha^2$.

Let us recall the definition of Riesz basis.

\begin{definition}
	A system $\{f_n\}_n \subset H$ is a \textbf{Riesz basis} for $H$ if and only if
	\begin{itemize}
		\item[a)] every $f \in H$ can be written in the form $ f = \sum_n c_n f_n$;
		\item[b)] there exist constants $A$ and $B$,  $0< A \le B $, such that
		\begin{equation}\label{RieszBasis}
		A \left( \sum_n |c_n|^2 \right)^{\frac{1}{2}} \le \norm[\Big]{\sum_n c_n f_n} \le B \left( \sum_n |c_n|^2 \right)^{\frac{1}{2}}
		\end{equation}		
		for every finite sequence of scalars $c_n$.
	\end{itemize}
\end{definition}

The theorem of Hedenmalm, Lindqvist and Seip mentioned above states as follows.

\begin{theorem} \label{teoseip} \normalfont{\cite[Theorem 5.2]{HeLiSe}}
	The system $\{\varphi_n\}$ is a Riesz basis of $L^2(0,1)$ if and only if $S\varphi$ and $1/S\varphi$ belong to $\mathcal{M}.$
\end{theorem}

Furthermore, in \cite{HeLiSe} the authors give a characterization of the multipliers of $\mathcal{H}^2$ as a subspace of $\mathfrak{D}$. Let us first define the Banach algebra
\[
	\Ha^{\infty} =\Big\{ \sum a_n n^{-s} \text{ which defines a bounded holomorphic function on } \mathbb{C}_{0}  \Big\},
\]
where $\C_{\sigma}=\{  s \in \C : \operatorname{Re}(s)>\sigma \}.$

This space is endowed with the norm
\[
	 \norm[\Big]{\sum a_n n^{-s}}_{\Ha^{\infty}} = \sup\limits_{ \operatorname{Re} s >0}  \left| \sum a_n  n^{-s} \right|
\]
and identifies with the space of multipliers of $\mathcal{H}^2$, as the following theorem shows.

\begin{theorem}\label{Multipliers} \normalfont{\cite[Theorem 3.1]{HeLiSe}}
 Let $D$ a Dirichlet series in $\mathfrak{D}$. Then $D$ belongs to $\Ha^{\infty}$ if and only if $M_D: \Ha^2 \rightarrow \Ha^2$ is well-defined and continuous. 	
In this case we have
	
$$\norm{D}_{\mathcal{H}^{\infty}} = \norm {D}_ {\mathcal{M}} = \sup\limits_{\norm{E}_{\Ha^2} < 1} \norm{D * E}.$$
In other words, we have that $\mathcal{M} = \mathcal{H}^{\infty}$.

\end{theorem}

\subsection{The spaces $ \mathcal{H}^{\infty}$, $H^{\infty}(\mathbb{T}^{\infty})$ and $H^{\infty}(B_{c_0})$ }

In this sections we describe    three spaces that have a central role in this work.
The idea of Bohr was to link the formal power series in infinitely many variables to Dirichlet series (also viewed as formal series) as follows. Let $\primos=\{ p_n \}_n$ be the ordered sequence of prime numbers. By the fundamental theorem of arithmetic each $n \in \N$ has a unique representation of the form $n= \primos^{\alpha} = p_1^{\alpha_1}\cdot \cdots \cdot p_m^{\alpha_m}$, this enables us to define
\begin{equation}\label{alphan}
\alpha (n) = (\alpha_1 , \ldots, \alpha_m, 0, \ldots) \in \N_0^{(\N)}.
\end{equation}
Bohr defined the bijection
\begin{align*}
\mathfrak{B}: \mathfrak{P} &\xrightarrow{\phantom{ a_n = a_{\primos^{\alpha}}=c_{\alpha}}} \mathfrak{D} \\
\sum_{\alpha} c_{\alpha}  z^{\alpha}&\xrightarrow{ a_n = a_{\primos^{\alpha}}=c_{\alpha}} \sum_{n=1}^{\infty} a_{n}n^{-s}.
\end{align*}
between the space $\mathfrak{P}$ of power series and the space $\mathfrak{D}$ of Dirichlet series. The function $\mathfrak{B}$ is called the Bohr transform. Furthermore, $\mathfrak{B}$ is linear and  multiplicative transform, i.e., it is an algebra isomorphism.

The linear space of all bounded and holomorphic functions $f$ from $B_{c_0}$ to $\C$ is denoted $H^{\infty}(B_{c_0})$ and endowed with the supremum norm.

The following theorem can be found in \cite[Theorem 3.8]{DirSer} and is crucial for what comes next.

\begin{theorem}\label{H inf DirSer} The following equality holds
	\[
	H^{\infty} (B_{c_0}) = \mathcal{H}^{\infty}.
	\]
	The Bohr transform is the unique isometric and linear bijection from $H^{\infty}(B_{c_0})$ into $\mathcal{H}^{\infty}$ wich assigns to each $f$ the Dirichlet series $\sum a_n n^{-s} = \mathfrak{B}f$, i.e. \break $a_n= c_{\alpha}(f)$ for all $n \in \N$ and $\alpha \in \N_0^{(\N)}$ with $n= \primos^{\alpha}$. Moreover, for each $s \in \C_{0}$
	\begin{equation}\label{TransfB}
	f\left( \frac{1}{p^{s}}\right) = \sum_{n=1}^{\infty} a_n n^{-s}.
	\end{equation}
	
\end{theorem}

It is important to highlight that this theorem implies the following remark.

\begin{remark}\label{unicidadC00}
	Every $f$ in $H^{\infty}(B_{c_0})$ has a unique monomial series representation of the form
	$f(z) \sim \sum_{\alpha} c_{\alpha} z^{\alpha} $ which coincides with $f$ on each $z\in B_{c_0}$ with only finitely many nonzero elements. The diagram \ref{diag} above illustrates this idea.
\end{remark}

In addition to these, they also introduced the Fourier analysis point of view.  Before we proceed let us recall that for a function $ f \in L^1(\mathbb{T}^{\infty})$ and $\alpha \in \mathbb{Z}^{(\mathbb{N})}$ , the  Fourier $\alpha-$coefficient is defined as
\[
\widehat{f}(\alpha) = \int\limits_{\mathbb{T^{\infty}}} f(w) w^{-\alpha} dw.
\]

Now we can define the Hardy space
\begin{equation*}
H^{\infty}(\mathbb{T^{\infty}}) = \{ f \in L^{\infty}(\mathbb{T^{\infty}}): \widehat{f}(\alpha) = 0 \textit{ if } \alpha_j \not\geq 0 \textit{ for every } j=1, 2, \ldots \}.
\end{equation*}

It is known that a function $f \in L^1(\mathbb{T^{\infty}})$ is uniquely determined by its  Fourier coefficients (see, for example,   \cite[Theorem 5.17]{DirSer}). It is also proved there  that the following three settings (or `worlds') are basically the same (see  \cite[Corollary 5.3]{DirSer}).
\begin{diagram}
\begin{center}
	\begin{tikzpicture}
	\matrix (m) [matrix of math nodes, row sep=0.1em, column sep=5.5em]
	{ \,  & \mathfrak{P}  & \mathfrak{P}  & \mathfrak{D} & \, \\
		\,  & \cup  & \cup & \cup & \, \\
		\,~ & H^{\infty}(\mathbb{T^{\infty}})  & H^{\infty}(B_{c_0}) & \mathcal{H}^{\infty} &  \\ };
	\end{tikzpicture}
	\begin{tikzpicture}
	
	\matrix (m) [matrix of math nodes, row sep=0.5em, column sep=5em]
	{ \, & \sum_{\alpha} \hat{f}(\alpha) w^{\alpha} &  \sum_{\alpha} c_{\alpha}  z^{\alpha} & \sum_{n} a_n n^{-s} &\, \\ };
	{ [start chain]
		\chainin (m-1-2);
		\chainin (m-1-3) [join={node[above,labeled] {c_{\alpha}=\hat{f}(\alpha)}}];
		\chainin (m-1-4) [join={node[above,labeled] {\qquad a_{n}=a_{\primos^{\alpha}}=c_{\alpha} \qquad}}];
	}
	\end{tikzpicture}
\end{center}
\caption{Relationship between the different $H^\infty$ spaces.}
\label{diag}
\end{diagram}

This is a crucial equality because it allows us transcribe our questions  from one place to another in searching for solutions. More precisely, \cite[Corollary 5.3]{DirSer} reads as the first assertion of the following corollary. The second assertion is much simpler since the Bohr transform maps the orthonormal basis $\{z^{\alpha}\}_{{\alpha} \in {\N}^{(\N)}_0}$ for $H^2 (\T^{\infty})$ into $\{n^{-s}\}_n$ an orthonormal basis for $\Ha^2$.
\begin{corollary}\label{3settings}
	The following isometric Banach space equalities hold
	\[
	\mathcal{H}^{\infty}(\mathbb{T^{\infty}}) = H^{\infty}(B_{c_0}) = \mathcal{H}^{\infty},
	\]
	identifying monomial, Fourier and Dirichlet coefficients. The same result holds for $H^2(\mathbb{T}^{\infty})$, that is
		\[ \mathcal{H}^{2}(\mathbb{T^{\infty}}) = H^{2}(B_{c_0}) = \mathcal{H}^{2}
		\]
	isometrically.
	\end{corollary}

\section{Bessel sequences  and lower frame bounds }\label{SBesselseq}

Our goal is to split Theorem \ref{teoseip} into two parts. More precisely, we aim to characterize whenever a system satisfies one  of the inequalities from \eqref{frame} in terms of the space of multipliers $\mathcal{M}$.

Let us recall some definitions related to systems in Hilbert spaces $H$.

\begin{definition}
	A sequence $\{\varphi_n\}_n \subseteq H$ is a \textbf{frame} for $H$ if there exist $A, B >0$ such that  for all $f \in H$  we have
	\begin{equation}\label{frame}
	A  \|f \|^2 \le \sum_n \big| \langle f, \varphi_n \rangle \big|^2 \le B \|f \|^2.\end{equation}
	At the same time, if only the upper bound holds $\{\varphi_n\}_n$ is said to be a \textbf{Bessel sequence}.
	
\end{definition}
\begin{definition}  Let $\varphi_n$ be a Bessel sequence in a Hilbert space H. The operator  $C: H \rightarrow \ell^2$  defined by $C(f) = \{\langle f, \varphi_n \rangle\}_n$ is called the \textbf{analysis operator} for $\{\varphi_n\}_n$,  and its adjoint $C^*: \ell^2 \rightarrow H$ given by
	$$
	C^*: \{c_n\}_n \mapsto \sum_n c_n \varphi_n,
	$$
	is known as the \textbf{synthesis operator}.
	
\end{definition}

Now we  discuss some simple and useful reformulations for a sequence being Bessel. It is well-known that  $\{\varphi_n\}_n $ is a Bessel sequence if and only if $C^{*}$ is a well-defined bounded operator from $\ell^2$ to $H$. See \cite{Ch2016} for a more detailed presentation.

In our case, the adjoint of $T_{\varphi}$ is given formally (without considering the domain) by the following calculation:
\begin{align*}
\langle g , T_{\varphi} f\rangle = \langle g ,   \sum_{n} \langle f, e_n \rangle \varphi_n \rangle  = \sum_n \overline{\langle f, e_n \rangle} \langle g , \varphi_n \rangle = \langle \sum_n \langle g , \varphi_n \rangle e_n , f \rangle.
\end{align*}
From which  we can deduce that
\begin{equation*}
T^*_{\varphi} (g)= \sum_{n} \langle g, {\varphi}_n \rangle e_n.
\end{equation*}
Then it follows easily from the orthogonality of $\{e_n\}$ that
\begin{equation}
\Big\| T^*_{\varphi} (g) \Big\|= \sum_n \big|  \langle g, {\varphi}_n \rangle  \big|^2.
\end{equation}

\begin{remark}\label{obsC*}
	
	Notice that our case, that is when $\{ \varphi_n \}_n$is a system of dilated functions, we have that the synthesis operator $C^*$  is quite similar to $T_{\varphi}$. Indeed, if we define the isometry $L: \ell^2  \rightarrow L^2 (0,1)$ given by $L(\{c_n\}_n) = \sum_n c_n n^{-s} $, we get the relation
	\begin{equation}
	T_\varphi \circ L = C^*.
	\end{equation}
	This allows us to  deduce that $\{\varphi_n\}_n$ is a Bessel sequence if and only if $T_{\varphi}$ is well-defined and bounded.
	
\end{remark}

We see from Theorem \ref{teo1} below that $\{\varphi_n\}_n$ is a frame if and only if $S\varphi$ and $1/ S\varphi$ both belong to $\mathcal{M}$. Then, by Theorem \ref{teoseip}, this means that the system $\{\varphi_n\}_n$ is a frame if and only if it is a Riesz basis.  We  may also deduce this result from the fact that every frame and minimal sequence  turns out to be a Riesz basis. It is shown in \cite{HeLiSe}  that $\{\varphi_n\}_n$ is  a minimal system. They constructed  a biothogonal system for $\{\varphi_n\}$; that is, a system $\{\psi_n \}$ such that $\langle  \varphi_j , \psi_k  \rangle= \delta_{j,k} $ for all $k$ and $j$.

 \begin{theorem}\label{teo1} Let $\varphi \in L^2 {(0,1)}$ and $\varphi_n (x) = \varphi(nx)$. Then the following assertions hold:
 	\begin{itemize}
 		\item The system $\{\varphi_n\}$ is a Bessel sequence if and only if $S\varphi \in \mathcal{M} $.
 		\item The system $\{\varphi_n\}$ satisfies the lower frame bound  if and only if $1/S\varphi \in \mathcal{M} $. 	\end{itemize}
	
 \end{theorem}

The central difficulty in the proof of this theorem is in the second part, and it is related to the fact that the operator $T_{\varphi}$ need not be bounded.

\begin{proof}
By  Remark \ref{obsC*}, $\{\varphi_n\}$ is a Bessel sequence if and only if  $T_{\varphi}$ is  a well-defined and bounded operator.  By the identity \begin{equation}\label{eq-TfiM-conj}
 T_{\varphi} = S^{-1} \circ M_{S\varphi}\circ S,
\end{equation}
this is equivalent to  $M_{S\varphi}$ being well-defined and bounded, i.e., to $S\varphi$ being a multiplier.  Moreover, the best Bessel constant is given by $\norm{T_{\varphi}}= \norm{M_{S\varphi}}$ and by Theorem \ref{Multipliers} it is also equal to  $\norm{S\varphi}_{\mathcal{H}^\infty}$.

\smallskip
For the proof of the second assertion, we use again the fact that  the operators $T_{\varphi}$ and $M_{S\varphi}$ are conjugated (see  \eqref{eq-TfiM-conj}).
We consider \begin{equation}
M_{S\varphi}:  \Dom(M_{S\varphi}) \subseteq \mathcal{H}^2 \rightarrow \mathcal{H}^2,
\end{equation}
 where $\Dom(M_{S\varphi})$ is the natural domain for a multiplication operator. That is, all the Dirichlet series $D \in \Ha^2$ such $S\varphi * D \in \Ha^2$.

Note that, since $S\varphi$ belongs to $\mathcal{H}^2$, $\Dom(M_{S\varphi})$ contains $\mathcal{M} = \mathcal{H}^{\infty}$ (this is a consequence of Theorem~\ref{Multipliers}). This means that $M_{S\varphi}$ is {densely defined}, since $\mathcal{H}^{\infty}$ is dense in $\mathcal{H}^2$.

Let us see that $M_{S\varphi}$ is also {closed}:
take $D_N \in \Dom(M_{S\varphi})$ such that
 	\begin{equation*}
	 	D_N \xrightarrow[]{\mathcal{H}^2} D \quad \textit{ and } \quad  M_{S\varphi}(D_N)  \xrightarrow[]{\mathcal{H}^2} \widetilde{D}.	\end{equation*}
 	We need to show that $ M_{S\varphi}(D) = \widetilde{D}$. The convergence in $\mathcal{H}^2$  implies that  $D_N$ converge pointwise to $D$ in  $\mathbb{C}_{\frac{1}{2} + \varepsilon}$  for every $\varepsilon >0$. We fix any such  $\varepsilon$, so we have that $	
		D_N (s) \rightarrow D(s)$  for all $s \in \mathbb{C}_{\frac{1}{2} + \varepsilon}$ which implies that 	$
 S\varphi(s) D_N (s) \rightarrow  S\varphi(s) D(s)$ for all $s \in \mathbb{C}_{\frac{1}{2} + \varepsilon}$.  Since this holds for every $\varepsilon >0$ and we also have that $M_{S\varphi}(D_N)  \xrightarrow[]{} \widetilde{D}$ pointwise, it  follows that $ S\varphi D = \widetilde{D}$, as desired.

\smallskip
 Consider the following operator $ Q = S^{-1} \circ M_{S\varphi} \circ S $, with
$$
 \Dom(Q)= S^{-1}(\Dom(M_{S\varphi})).
$$
Also consider $R : \Dom(R) \subset L^2 {(0,1)} \rightarrow L^2 {(0,1)} $
$$
 R(f) = \sum \langle f , \varphi_n \rangle e_n ,
$$
where $\Dom(R)=  \{ f \in L^2(0,1): \sum |\langle f , \varphi_n \rangle |^{2} < \infty \}$. We claim that $R=Q^{*}$. Recall that $f \in \Dom(Q^{*})$ if and only if
 $$ g \mapsto \langle Qg , f \rangle $$ is a continuous map for all $g \in \Dom(Q)$. In this case, we can write
\begin{align*}
 \langle Qg , f \rangle &=  \langle S^{-1} \circ M_{S\varphi} \circ Sg , f \rangle = \langle M_{S\varphi} \circ Sg , Sf \rangle \\
&= \sum_n \left( \sum_{kl = n} \langle \varphi , e_k \rangle \langle g , e_l \rangle \right) \overline{\langle f , \varphi_n \rangle }\\
&= \sum_n  \langle g , e_n \rangle  \overline{\langle f , \sum_k  \langle \varphi , e_k \rangle e_{kl}  \rangle}\\
&=  \langle g , \sum_n \langle f, \varphi_n \rangle e_n \rangle  =  \langle g , Rf \rangle.
\end{align*}

Then,  $g \mapsto \langle Qg , f \rangle $ is continuous if and only if $\sum_n \langle f, \varphi_n \rangle e_n$ belongs to $ L^2 (0,1)$ or, equivalently, $ \sum |\langle f , \varphi_n \rangle |^{2} $ is finite.  This, together with the previous calculations, shows that  $\Dom(R)=\Dom (Q^{*})$ and $R= Q^*.$

Recall that $\{\varphi_n\}_n$ satisfies the lower frame bound if there exists $A>0$ such that
\begin{equation*}
A \norm{f}^2 \le \sum |\langle f , \varphi_n \rangle |^{2} \quad \forall\, f \in L^2 (0,1).
\end{equation*}
Then, it is clear that $\{\varphi_n\}_n $ satisfies the lower frame bound if and only if the operator $R=Q^*$ is bounded below.

Some properties of bounded operators generalize to closed densely defined operators. For example, a variant of the closed range theorem is valid. Applying this to $Q$ (we have seen that $Q$ is a closed densely defined operator), we get that $Q$ is surjective if and only if $Q^*$ is bounded below.

Finally, we need to prove that the operator $Q$ is surjective if and only if  $1/S\varphi \in  \mathcal{M}$.  Assumed that $Q$ is surjective and take $D \in \Ha^2$. Since $M_{S\varphi}$ is also suryective, there exists $E \in \Dom(M_{S\varphi})$ such that
$$
M_{S\varphi} (E) = S\varphi * E = D
$$

Then, we have
$$
\frac{1}{S\varphi} * D = \frac{1}{S\varphi} * S\varphi * E = E \in \Ha^2.
$$
Therefore,  $1 / S\varphi * D \in \Ha^2$ for every $D \in \Ha^2$ and $  1/S\varphi \in \mathcal{M}$. Suppose now that $1/S\varphi$ is a multiplier and take $E \in \Ha^2$. We have that
$ \frac{1}{S \varphi} * E $ belongs to $\Ha^2$ and
$$
M_{S\varphi} \left(\frac{1}{S \varphi} * E \right) = E,
$$
and this proves that $Q$ is surjective.
\end{proof}

\begin{remark}\label{ciclicos}
It is known that a function $F \in  \Ha^2$ such that $1/F \in  \Ha^\infty$ is cyclic for $\Ha^2$ (see \cite{HeLiSe} or \cite{NI12}). In this direction, the second part of Theorem \ref{teo1} gives a new characterization of this type of cyclic vectors for $\Ha^2$.
\end{remark}

\section{Riesz sequences}\label{Srieszseq}

In the previous section we studied the characterization from \cite{HeLiSe} of Riesz bases of dilations. In Theorem \ref{teo1} we weakened the assumption of $\{\varphi_n\}$ being a Riesz basis to satisfying only one of the frame inequalities. Another natural way to weaken the Riesz basis condition is to assume that $\{\varphi_n\}$ is just a Riesz sequence.
\begin{definition}\label{sucriez} A sequence $\{\varphi_n\}_n \subset H$ is a \textbf{Riesz sequence} if it is a Riesz basis for its closed span in H. In other words, a Riesz sequence satisfies \eqref{RieszBasis} but is not necessarily complete in  $H$.
\end{definition}

In our case, a look at  \eqref{RieszBasis} shows that the system of dilated functions $\{\varphi_n\}_n \subset L^2 (0,1)$ is a Riesz sequence if and only if there exist $A, B >0$ such that
$$
A \| f\| \le \| T_{\varphi} f \| \le B \| f \|,
$$
for every finite sum $f= \sum c_n e_n$.  Since $e_n$ is a basis for $L^2(0,1)$, we deduce that $\{\varphi_n\}_n$ is a Riesz sequence if and only if the operator $T_{\varphi}$ is bounded and bounded below.  Also, since $T_{\varphi} = S^{-1} \circ M_{S\varphi}\circ S$, this is also equivalent to  $M_{S\varphi}: \Ha^2 \rightarrow \Ha^2 $ being bounded and bounded below.

We want to determine conditions for $\{\varphi_n\}$ to be a Riesz sequence in terms of the Dirichlet series  $S\varphi$. Looking at the results in  \cite{HeLiSe} and the previous section, one may be tempted to guess that  $M_{S\varphi}$ is bounded below if and only if there exists some $\varepsilon>0$ such that $|S \varphi(s)|> \varepsilon $ for all $s \in \mathbb{C}_{0}$. The following simple example shows this is not the case.
\begin{example}\label{ex-nobdedbelow}
	If we consider $\varphi(x)= e_2(x)=\sqrt{2}\sin(2\pi x) $, then $\{\varphi_n\} $ is clearly a Riesz sequence (in fact, it is an orthonormal sequence). However,
	$
	S\varphi(s)=  \frac{1}{2^s}$, which is not bounded below in $ \mathbb{C}_{0}.
	$
\end{example}

To discuss the main theorem of this section, we need to introduce various concepts regarding characters of compact group and Dirichlet series. For more details about these topics see \cite{RUD62} and \cite{HeLiSe}.

\subsection{Horizontal translations and Fatou's Theorem}

Recall that a function $\gamma: \N \rightarrow \mathbb{C}$ is a \textbf{\textit{character}} if it satisfies:
\begin{itemize}
	\item $\gamma (mn)= \gamma (m) \gamma(n)$ for all $n, m \in \N$;
	\item $|\gamma(n)|= 1$ for all $n \in \N$.
\end{itemize}
We denote by $\Xi$ the set of all such characters.

Given a Dirichlet series $D = \sum a_n n^{-s}$, each character $\gamma \in \Xi$ defines a new Dirichlet series by
\begin{equation}\label{eq-Dchi}
  	D^{\gamma} (s) = \sum a_n \gamma(n) n^{-s}.
\end{equation}
In particular, for   $\gamma(n)=n^{-it}$ we get
$$
D^{\gamma}(s)= \sum_n a_n n^{-it} n ^{-s} = D(s + it)
$$
which is called a \textbf{\emph{horizontal translation}} of $D$.

Characters $\gamma \in \Xi$ identify with elements $w\in \mathbb{T^{\infty}}$, taking $w=(\gamma(p_1), \gamma(p_2), \dots)$.
With this in mind, we rewrite the Dirichlet series in \eqref{eq-Dchi} as
\begin{equation}\label{eq-Domega}
 D^w (s)= \sum a_n w^{\alpha(n)} n^{-s},
\end{equation}
with  $w$ as above .

Note that if $f= \mathfrak{B}^{-1}(D)$, by definition of the Bohr transform we have
\begin{equation}\label{relDw}
\mathfrak{B}^{-1}(D^{w})(u)=f(w \cdot u).
\end{equation}
Now we turn to our desired characterization and assume that the system  $\{\varphi_n \}_n$ is a Riesz sequence. By Theorem \ref{teo1}, we know that $D= S\varphi$ belongs to  $\mathcal{M}  =\mathcal{H}^{\infty}$. Then, its Bohr antitransform $f=\mathfrak{B}^{-1}(D)$ belongs to  $H^{\infty}(\mathbb{T^{\infty}})$. By \eqref{relDw} we have that for each  $ w\in \mathbb{T}^{\infty}$ the function $\mathfrak{B}^{-1}(D^{w})$ is also in $H^{\infty}(\mathbb{T^{\infty}})$. This means that the Dirichlet series $D^w$  also  belongs to $\Ha^{\infty}$ for every  $ w\in \mathbb{T}^{\infty}$. Then, an adaptation of Fatou's Theorem for Dirichlet series in $\C_{>0}$ guarantees that for all  $w \in \mathbb{T}^{\infty}$ the limit
\begin{equation}
\lim\limits_{\sigma \to 0^{+}} D^{w}(\sigma + i t) \text{ exists for almost all } t.
\end{equation}
Details of this can be found in \cite[Lemma 11.22]{DirSer}.

As observed in \cite[Theorem 2]{SaSe}, this allows us to pick a representative $\widetilde{f}$ of $f$ in $H^{\infty}(\mathbb{T^{\infty}})$ which satisfies
\begin{align}\label{rep}
\widetilde{f}(w) =
\begin{cases}
\lim\limits_{\sigma \to 0^{+}} D^{w}(\sigma)  &\quad\text{if the limit exists;} \\
0 &\quad\text{otherwise.}\\
\end{cases}
\end{align}
We present some details for the sake of completeness. Let $$A=\{w \in \mathbb{T^{\infty}}: \lim\limits_{\sigma \to 0^{+}}  D^w(\sigma) \text{ exists.}\}.$$  Let us see that $|A|=1$. Since $\chi_A \in  L^1({\mathbb{T}}^{\infty}) $ we can apply the Birkhoff-Khinchin ergodic Theorem \cite{BIR31} to get that for almost all $w_0 \in \mathbb{T}^{\infty}$ we have
$$
 \int\limits_{\mathbb{T^{\infty}}} \chi_{A}(w) \, dw = \lim\limits_{R \to \infty} \frac{1}{2R} \int\limits_{-R}^{R}\chi_A (T_t(w_o)) \, dt.
$$
Here  $T_t : \mathbb{T^{\infty}} \rightarrow \mathbb{T^{\infty}}$ is the map $T_t(w) = (\primos^{-it} \cdot w)$. The system  $\{T_t\}_t$ is known as the Kronecker flow and its ergodicity can be found \cite{ERGODIC12}.

Notice that $T_t(w_0)$ belongs to $A$ if and only if the $\lim\limits_{\sigma \to 0}  D^{T_t(w_0)}(\sigma)$ exists. Applying Fatou's Theorem for Dirichlet series we deduce  that $T_t(w_0)$ belongs to $A$ for almost all $t$ and consequently $|A|=1$. This proves that $\widetilde{f}$ belongs to $H^{\infty}(\mathbb{T}^{\infty})$ and, to see that $\widetilde{f}$ is a representative of $f$, it is enough to compare their Fourier coefficients (see \cite[Theorem 2]{SaSe}). From now to the end of this section,  $f$  is  always the representative that satisfies \eqref{rep}.

Let us fix the following notation:
$$	D^{w} (it_0) := \lim\limits_{\sigma \to 0^{+}} D^{w}(\sigma +it_0).$$
Since $f$ satisfies \eqref{rep}, taking $t_0= 0$ we get that for almost all $w \in \mathbb{T^{\infty}}$
$$f(w) = D^{w}(0).$$
Moreover, given $t_0 \in \mathbb{R}$ we have
\begin{equation}\label{relaciont0}
 D^{w}(it_0)= \lim\limits_{\sigma \to 0^{+}} D^{w}(\sigma +it_0) =\lim\limits_{\sigma \to 0^{+}} D^{T_{t_0}(w)}(\sigma) = f (T_{t_0}(w)).
\end{equation}
From this identity we derive the next remark.

\begin{remark}\label{equivDw} The following conditions are equivalent.	
	\begin{enumerate}
	\item[$i)$] There exists $\tilde{t}_0$ such that $
	|D^{w}(i\tilde{t}_0)| \geq \varepsilon $ for almost every $ w  \in \mathbb{T^{\infty}}.$
	\item[$ii)$] For all $t_0$ there exists $B_{t_0} \subset \mathbb{T^{\infty}}$ with total measure  such that $
	|D^{w}(it_0)| \geq \varepsilon $ for all $ w  \in B_{t_0}.$
	\end{enumerate}
\end{remark}
Obviously , $ii)$ implies $i)$. For the converse, take $t_0 \in \mathbb{R}$ and set $$B_{t_0} = \{\primos^{-i(t_0 + \tilde{t_0})} \cdot w : w \in B_{\tilde{t_0}}\},$$
which is clearly of total measure.
Take $w' \in B_{t_0}$ and choose $w \in B_{\tilde{t}_0}$ such that $w'=\primos^{-i(t_0 + \tilde{t_0})} \cdot w$, thanks to \eqref{relaciont0} we get that
$$
  |D^{w'}(it_0)| = |f(T_{\tilde{t}_0}(w))|  \geq \varepsilon .$$

Now we are ready to state the main result of the section.

\subsection{Main results and applications}
The remainder of this section is devoted to state and prove our main result.
 \begin{theorem}\label{teo2} Let $\varphi \in L^{2}(0,1)$ and set $D= S\varphi$ and $f=\mathcal{B}^{-1}(D)$. The following statements are equivalent.
 	\begin{itemize}
 	\item[$i)$] The system $\{\varphi_n\}_n$ is a Riesz sequence.
 	\item[$ii)$] The function $f$ belongs to $H^{\infty}(\mathbb{T^{\infty}})$ and there exist $\varepsilon>0$ and a set $B \subseteq \mathbb{T^{\infty}}$ of total measure such that
 	\begin{equation}\label{acinff}
 	|f(z)| \geq \varepsilon, \qquad \textit{for all }z \in B \subset \mathbb{T^{\infty}}.
 	\end{equation}
 	\item[$iii)$] The Dirichlet series $D$ belongs to $\mathcal{H}^{\infty} $ and there exists $\varepsilon >0$ such that for almost all  $(\gamma, t) \in \Xi \times \mathbb{R}$ we have
 	\begin{equation}\label{Dgeqep}
 	|D^{\gamma} (it)| \geq \varepsilon.
 	\end{equation}
 	
 	\end{itemize}

In this case,
\begin{equation}
\varepsilon \left(\sum |b_n|^2  \right) ^{1/2} \le \norm[\Big]{\sum b_n \varphi_n}_{L^2} \le \norm{D}_{\mathcal{H}^{\infty}} \left(\sum |b_n|^2  \right) ^{1/2},
\end{equation}
for every finite sequence of scalars.
\end{theorem}

\begin{proof} It follows from Theorem \ref{teo1} that $\{\varphi_n\}_n$ satisfies the right-hand inequality from \eqref{RieszBasis} if and only if $D$ belongs to $\mathcal{H}^{\infty}$, which in turn is equivalent to $f$ being in $H^{\infty}(\mathbb{T^{\infty}})$. So we have to show the part involving the lower bounds.
	
It is clear from the definitions that  $\{\varphi_n\}_n $ is a Riesz sequence if and only if the operator $T_{\varphi}$ is  bounded and bounded below. This is  equivalent to $M_f: H^2 (\mathbb{T^{\infty}}) \rightarrow H^2(\mathbb{T^{\infty}})$ being bounded and bounded below.
	
As a first step, we  show that the bounded operator $M_f:H^2(\mathbb{T^{\infty}}) \to H^2(\mathbb{T^{\infty}})$ is bounded below if an only if there exists a set $B \subseteq \mathbb{T^{\infty}}$ of total measure on which $|f|$ is bounded below. The \emph{if} part is clear,
so assume that  $M_f$ is bounded and bounded below with constant $\varepsilon$ and consider  the set
$$
N= \{z \in \mathbb{T^{\infty}} : |f(z)| <  \varepsilon \}.
$$	
Set $\chi_{N}$ to be the characteristic function of $N$. Since trigonometric polynomials are dense in $L^2 (\mathbb{T^{\infty}})$ (see \cite[Proposition 5.5]{DirSer}) there exist a sequence  $P_k$ of polynomials of degree $n_k$ in $N_k$ variables (in  $z$ and $\overline{z}$) such that
\begin{equation}\label{approxchar}
\lim_{k} P_k = \chi_N \quad \text{in } L^2 (\mathbb{T^{\infty}}).
\end{equation}
A simple calculation shows that
$$
 |N| =  \norm{\chi_N}^2_{L^2} = \lim_{k}  \norm{P_k}^2_{L^2}  =  \lim_{k}  \norm{z^{n_k}_1 \ldots z^{n_k}_{N_k} P_k}_{L^2}^2 .
$$
Now we have that the polynomials  $z^{n_k}_1 \ldots z^{n_k}_{N_k} P_k$ belong to $H^2 (\mathbb{T^{\infty}})$. Our assumptions on $M_f$ give
\begin{eqnarray*}
  \varepsilon^2 |N| &\le & \liminf_k \norm{M_f (z^{n_k}_1 \ldots z^{n_k}_{N_k} P_k)}^2= \norm{f\, \chi_N}^2_{L^2} \\
   &=& \int\limits_{N} |f|^2  dz
\end{eqnarray*}
Since $|f|<\varepsilon$ on $N$, this implies that $|N|=0$, so the set $B=N^c$ satisfies \eqref{acinff}.

To end the proof, we  show next that \eqref{Dgeqep} is equivalent to \eqref{acinff}.  If \eqref{acinff} holds, by the definition of $f$ we have $|f(w)|=\lim\limits_{\sigma \to 0} |D^{w}(\sigma)|\geq \varepsilon$ for $w \in B$ ($|B|=1$). Combining this with Remark \ref{equivDw} we get  that the $t$-sections of the set
\[
C=\{(t, w) \in \mathbb{R} \times \mathbb{T^{\infty}}: |D^{w}(it)|< \varepsilon \},
\]
have zero measure. As a collorary of Fubini's Theorem we get that $C$ has measure zero.  The converse also follows from Fubini's Theorem.
\end{proof}

An immediate consequence of the previous theorem is the following characterization for orthonormal sequences of dilations.

\begin{corollary}\label{coro}
	The system $\varphi_n$ is an orthonormal sequence  if and only if $D \in \mathcal{H}^{\infty}$ and for almost all $\gamma \in \Xi$ and almost all $t \in \mathbb{R}$ we have $| D^{\gamma}(it)|=1.$
\end{corollary}

We mention that this corollary can also be proved by mimicking the proof of \cite[Proposition 5.1]{Olo10}: first we compute
\begin{align*}
 \langle n^{-s}S{\varphi}, m^{-s}S{\varphi} \rangle &= \int\limits_{\mathbb{T}^\infty} w^{\alpha} \mathfrak{B}^{-1}( S\varphi)(w) \overline{w^{\beta} \mathfrak{B}^{-1}( S\varphi)(w)} dw \\
 &= \int\limits_{\mathbb{T}^\infty} |\mathfrak{B}^{-1}( S\varphi)(w)|^2 w^{\alpha - \beta}  dw.
\end{align*}
This shows that the sequence $\{n^{-s} S{\varphi}\}_n$ is orthonormal if and only if $|\mathfrak{B}^{-1}( S\varphi)(w)|= |f(w)|=1$ for almost all $w \in \mathbb{T^{\infty}}$. Here we use the uniqueness of the Fourier coefficients. Now observe that
\[
n^{-s}S\varphi = S \varphi n^{-s} = S \varphi S e_n = S\circ T_{\varphi}(e_n) = S\varphi_n.
\]
Since $S$ is an isometric isomorphism, the previous equivalent conditions are also equivalent to   $\{\varphi_n\}_n$ being an orthonormal sequence, which gives Corollary \ref{coro}.

We conclude this section with a reformulation of a theorem from \cite{Olo10} in terms of orthonormal systems in $L^2(0,1)$. In order to state these results, we need to recall some definitions. By an  $\mathcal{H}^2-$inner function we mean a Dirichlet series $D$ in $\mathcal{H}^{\infty}$ such that  $|\mathfrak{B}^{-1}(D) (w) |= 1 $ a.e. $ w \in \mathbb{T^{\infty}} .$  For any $ n \in \N$ we have an operator $\Lambda(n)$ acting in $\Ha^2$ defined by
\[
\Lambda(n)f(s) = n^{-s} f(s), \quad \ s \in \C_{1/2},
\]
for $f \in \Ha^2$.

For the rest of this section,  $\mathcal{I}$ is a shift invariant subspace of $\mathcal{H}^2$ (i.e. $n^{-s} \mathcal{I} \subseteq \mathcal{I}$ for all $n \in \N$ ) such that
\begin{equation}\label{5.1}
\Lambda(p)|_{\mathcal{I}} \left( \Lambda(q) |_{\mathcal{I}}\right)^* =\left(\Lambda (q)|_{\mathcal{I}}\right) ^* \Lambda(p) |_{\mathcal{I}}
\end{equation}
for all primes $p, q$ with $ p \neq q$.  This  property  is Condition 2 in \cite{Olo10}. Now we can state the following result.

\begin{theorem} \normalfont{\cite{Olo10}} \label{teolof}
	Let $\mathcal{I}$ be a nonzero subspace of $\mathcal{H}^2$ which is  shift invariant. Then  $\mathcal{I}$ has the property \eqref{5.1} for all primes $p$ and $q$ with $p \neq q$ is and only if  $\mathcal{I}$ is of the form
	\begin{equation}\label{phi}
		\mathcal{I} = D \mathcal{H}^2
	\end{equation}
for some  $\mathcal{H}^2-$inner function $D$.  Furthermore, the function $D$ is uniquely determined up to unimodular constant.
\end{theorem}
Our results allows us to give a reformulation this theorem. Given $D$ as in the theorem, set  $\varphi = S^{-1} D $. Combining Theorem \ref{teolof} with Corollary \ref{coro} we get that $\{\varphi_n\}_n$ is an orthonormal sequence. Thanks to the relationship between $T_{\varphi}$ and $S$ we deduce that equality \eqref{phi} turns into
\[
\mathcal{I} =  S\varphi S (L^2(0,1)) = S T_{\varphi}(L^2(0,1)).
\]
From this we conclude $S^{-1}(\mathcal{I})= \overline{[\varphi_n]}$.  Since this steps can be followed backwards then Theorem \ref{teolof} can be restated as follows.

\begin{theorem}
	Let $\mathcal{I}$ be a nonzero subspace of $\mathcal{H}^2$ which is  shift invariant. Then  $\mathcal{I}$  has the property \eqref{5.1} if and only if there exists $\varphi \in L^{2}(0,1)$ such that the associated system $\varphi_n$ is an orthonormal sequence and
	\[
	S^{-1}(\mathcal{I}) = \overline{[\varphi_n]}.
	\]
\end{theorem}

\section{Examples}\label{SExamples}
In this section we present examples and applications of the obtained results. More precisely, we show examples of $\varphi$ that give systems satisfying only one of the frame bounds, as well as systems which are Riesz or orthonormal sequences (or bases), etc.  In particular, we show nontrivial orthonormal sequences of dilations in  $L^2 (0,1)$ which are not block sequences of  $\{\sqrt{2}\sin(nx)\}_n$.

As usual, we write  $\varphi \in L^{2}(0,1)$ for the function that generates the system of dilations $\{\varphi_n\}$ and then set  $D= S\varphi$ and $f=\mathcal{B}^{-1}(D)$.

We start by recalling Example~\ref{ex-nobdedbelow}: the function $\varphi(x)= e_2(x)=\sqrt{2}\sin(2\pi x) $, for which we have $f(z)=z_1 \in H^{\infty}(\mathbb{T}^\infty)$ and   $D:=1/2^s$. The system $\{\varphi_n\}_n$ is clearly an orthonormal sequence (and then a Riesz sequence) but $D$ is not bounded below. On the other hand, it is easy to check that any of the conditions of Corollary  \ref{coro} are satisfied.

Now we present the argument that is the basis of most of our next examples.
Recall that every continuous function with no zeros in a compact set $X$, is bounded from below in that set.
If $f$ is a  polynomial in $N$ variables, then it clearly belongs to $H^{\infty}(\mathbb{D}^N) \subset H^{\infty}(B_{c_0})$.
If $f$ has no  zeros in $\mathbb{T}^N$, by compactness  $|f|$ is bounded below on $\mathbb{T}^N $ and by Theorem \ref{teo2} the corresponding system $\{\varphi_n\}$ turns out to be a Riesz sequence. If in addition $f$ has no zeros in  $\overline{\mathbb{D}}^{N}$, the multiplicative inverse $\frac{1}{f}$ belongs to $H^{\infty}(\mathbb{D}^N)$ and, therefore, $S\varphi$ y $1/S\varphi$ belong to $\mathcal{M}$. Applying Theorem  $\ref{teoseip}$ we conclude that  $\{\varphi_n\}_n$ is a Riesz basis.

We first consider $ \varphi (x)= c_1 c_2 \sqrt{2}\sin(x)- c_2 \sqrt{2}\sin(2x) - c_1 \sqrt{2}\sin(3x) + \sqrt{2}\sin(6x)$. Thus, the associated holomorphic function is
	$f(z)= (z_1 - c_1)(z_2 - c_2,)$ and the Dirichlet series  $$D(s)= c_1 c_2 - c_2 2 ^{-s} - c_1 3^{-s} + 6^{-s}.$$

In this and the general case in which 	$f(z)= (z_1 - c_1)(z_2 - c_2)\ldots (z_N - c_N)$, the associated system  $\{\varphi_n\}_n$ is a Riesz sequence if and only if $|c_i| \neq 1$  for all $i$. It is a  Riesz basis if and only if $|c_i | > 1$ for all $i$. Finally, the system is a Bessel sequence not satisfying the lower frame bound if and only if   $|c_i | \ge 1$ for all $i$ and there exists some $j$ with $|c_j | = 1$.

Consider now the $N-$variable polynomial
	\[
	f(z)= z_1 \ldots z_N - c .
	\]
The same argument shows that, for the corresponding function $\varphi$, $\{\varphi_n\}$ is a Riesz sequence if and only if $|c|\neq 1$, and a Riesz basis if and only if $|c| >  1$.

In  \cite{BOLU01}, the authors present conditions for an $N-$variable polynomial to have no zeros in $\mathbb{T}^N$ or $\mathbb{D}^N$. In order to apply their results, we need to introduce some notation. Consider a polynomial $f: \C^N \rightarrow \C$ given in Taylor expansion at the origin
\begin{equation}
f(z) = \sum _{k \in \mathbf{K}} b_k z^k, \qquad \mathbf{K} \subset \mathbb{N}^N_0 \quad y  \quad b_k \in \mathbb{C} \setminus \{0\}.
\end{equation}
The set $\mathbf{K}$ is the degree set and $V(\mathbf{K})$ the set of vertices which generates the convex hull of $\mathbf{K}$.

Let us consider a Dirichlet polynomial $D(s) = \sum_{n=1}^{6} a_n n^{-s}$.
and its associated polynomial through Bohr's transform
$$f (z) = a_1 + a_2 z_1 + a_3 z_2 + a_4 {z_1}^2 + a_5 z_3 + a_6 z_1 z_2.$$
We assume that  the coefficients $a_n$ are not zero, in which  case the degree set is
$$
\mathbf{K}(f)= \{ (0,0,0),\,(1,0,0),\, (0,1,0),\, (0,0,1),\ (2,0,0),\,(1,1,0)\}.
$$
Then, except for $(0,0,1)$, every other element of $\mathbf{K}(f)$ is in the $xy-$plane and they are distributed as shown in the following figure.

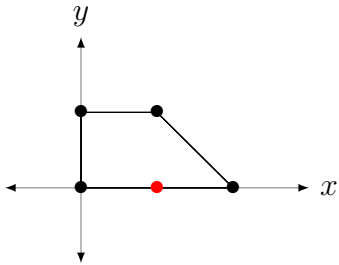
\begin{equation*}
\begin{tikzpicture}[baseline=(current  bounding  box.center)]
......	\draw[latex-latex, thin, draw=gray] (-1,0)--(3,0) node [right] {$x$}; 
\draw[latex-latex, thin, draw=gray] (0,-1)--(0,2) node [above] {$y$}; 

\foreach \Point in {(0,0), (0,1), (2,0), (1,1)}{
	\node at \Point {\textbullet};
	\draw   (0,1) -- (0,0)
	(0,0) -- (2,0)
	(0,1) -- (1,1)
	(1,1) -- (2,0);
}


\node [red] at (1,0) {\textbullet};;

\end{tikzpicture}
\end{equation*}

From this, it is immediate to see that  $(0,1,0)$ is not in $V(\mathbf{K})$ (indeed  we only need $a_1$, $a_2$ and $a_4$ to be nonzero).
Therefore, if
\begin{equation}\label{condBOLU}
|a_2| > \sum_{n \neq 2} a_n.
\end{equation}
the hypotheses of \cite[Corollary 1]{BOLU01} are satisfied. Then $f$  has no zeros in $\mathbb{T}^{\infty}$, and therefore, $\{\varphi_n\}$ is a Riesz sequence.

The previous example could be easily generalized to any Dirichlet polynomial $D(s) = \sum_{n=1}^{N} a_n n^{-s}$ which satisfies  \eqref{condBOLU} and such that $a_1$, $a_2$ and $a_4$ are nonzero, since in this case  the index in correspondence with $a_2$ is never a vertex (regardless of the number of variables).

We now consider a Dirichlet polynomial $D(s) = \sum_{n=1}^N a_n n^{-s}$ with $a_1  \neq 0$  and such that
\begin{equation}\label{condBOLU2}
|a_n| \le \left( 1 - \frac{1}{\left(1 + |a_1|\right)^{\frac{1}{\pi(N)}}}\right)^{|\alpha(n)|} \quad \text{for all} \quad n=1, \ldots, N.
\end{equation}
Here $\pi(N)$ is the well known prime-counting function, the function that counts the number of prime numbers less than or equal to some real number $N$, and $\alpha(n)$ is defined by the equation \eqref{alphan}. Note that $f$, the Bohr lift of $D$, is  a polynomial  in $\pi(N)$ many complex variables. Putting  \eqref{condBOLU2} together with \cite[Theorem 5]{BOLU01} we get that the system $\{\varphi_n\}$ is a Riesz basis for $L^2 (0,1)$.

We now leave the polynomial world and take $f$  as the product of $d$ Moebius transforms in different variables: given  $c_i \in \C \setminus\{0\}$ with $|c_i| < 1$  for all $i=1, \ldots
d$,  we define
$$f(z)= \frac{z_1-c_1}{1- \overline{c_1}z_1 } \ldots \frac{z_d-c_d}{1- \overline{c_d}z_d }.$$
Then, $f$ is a bounded  holomorphic  function with has modulus 1 in the polytorus. By Theorem \ref{teo2} we get that, for the corresponding $\varphi$, the sequence $\{\varphi_n\}$ is orthonormal in $L^2 (0,1)$.

In the case $d=1$,  it is easy to
rewrite the Moebius transform  to get
$$
f(z_1) =  \frac{z_1-c_1}{1- \overline{c_1}z_1 } = (z_1 - c_1) \sum_{m=0} \overline{c_1}^m z^m_1.
$$
Hence,  its associated mappings are
$$
D(s) = \overline{c_1} + \sum_{m=1}^{\infty} ( \overline{c_1}^{m-1} + c_1 \overline{c_1}^m) (2^m)^{-s}.
$$
and
\begin{equation}\label{ex-nosubsequence}
 	\varphi (x) = \overline{c_1}\sqrt{2}\sin(x) + \sum_{m=1}^\infty ( \overline{c_1}^{m-1} + c_1 \overline{c_1}^m) \sqrt{2} \sin(2^{m}x).
\end{equation}

\begin{example}
 Let  $\phi$ be as in \eqref{ex-nosubsequence} with $0<|c_1|< 1$. Then,   the system $\{\varphi_n\}_n$ is an orthonormal sequence of $L^2(0,1)$ which is not a block sequence $\{\sqrt{2}\sin(nx)\}_n$.
\end{example}

As a last example we construct a system $\{\varphi_n\}_n$ which satisfies the lower frame condition \eqref{frame} but not the upper one (i.e., it is not a Bessel sequence). By Theorem \ref{teo1} we need a function $\varphi \in L^{2}(0,1)$ such that  $D$ does not belong to $\Ha^{\infty}$ but $1/D$ belongs to $\Ha^{\infty}$.  To achieve this, we are going to work in the classical Hardy space on the disk $\D$. We are looking for a function $F \in H^{2}(\D)$ which is not in $H^{\infty}(\D)$ such that  $\frac{1}{F}$ belongs to $H^{\infty}(\D)$.  To do this we show that the outer function
\begin{equation}\label{Fext}
F(z) = \exp  \left( \frac{1}{2\pi}\int\limits_{-\pi}^{\pi} \frac{e^{i \theta} + z}{e^{i\theta}-z} \log|f(e^{i\theta})|d\theta \right),
\end{equation}
associated to $f(e^{i\theta}) = |\theta|^{-\frac{1}{3}}$  satisfies the requirements needed. By \cite[Theorem 17.16]{RU87} we get that $\lim\limits_{r \to 1^{-}} |F(re^{i\theta})|= f(e^{i\theta})$ a.e. in $\T$ and that $F$ belongs to $H^p(\D)$ if and only if $f$ belongs to $L^{p}(\T)$ $1\le p\le \infty$.  Since $f$ is in $L^2(\T)$  but not in $L^{\infty}(\T)$ we have that $F$ is in $L^2(\T)$ but not in $H^{\infty}(\D)$. A close look at \eqref{Fext} tells us that $\frac{1}{F}$ is the outer function associated to $\frac{1}{f}= |\theta|^{1/3}$ , which belongs to $L^{\infty}(\T)$. Consequently, applying \cite[Theorem 17.16]{RU87}  once again we get that $\frac{1}{F}$ belongs to  $H^{\infty}(\D)$. Thus, $ F $ fulfills the desired properties.

\section{Multivariate case}\label{SMultivariateCase}

The goal of this section is to study Riesz sequences,  frames and orthonormal bases in $ L ^ 2 ((0,1)\times (0,1)) $ of the form
\[
\varphi_{m,n} (x,y)= \varphi(mx,ny).
\]
The idea is to replicate the structure of Theorems \ref{teo1} and \ref{teo2} for $\varphi_{m,n}$. To do this, we need  the corresponding results of multiple Dirichlet series that we  develop next.

Now we transfer our questions to the context of multiple Dirichlet series.  We begin by defining them.

\begin{definition}
	An ordinary  \textbf{$k-$multiple Dirichlet series} is a series of the form
		\begin{equation*}
		 \sum_{m_1, \ldots, m_k =1}^{\infty} \frac{a_{m_1, \ldots, m_k}}{m_1^{s_1} , \ldots, m_{k}^{s_k}}
		\end{equation*}
where the coefficients $a_{m_1, \ldots, m_k}$ belong to $\C$ and $s_1, \ldots, s_k$ are complex variables.
\end{definition}

For a complete introduction and more information on the topic we refer to the Castillo Medina Ph.D. Thesis \cite{JC19} and the references therein.

 For the sake of simplicity,  we work with double Dirichlet series. The extension to the general case is straightforward. Let  $I^2$ denote the set $(0,1)\times (0,1)$. Given a function $\varphi \in L^2 (I^2)$ we consider its dilations in both variables
\[
\varphi_{m,n} (x,y)= \varphi(mx,ny).
\]
Recall that  $L^2 (I^2)$ has a basis consisting of the functions
\[
e_{m,n} (x,y) = e_m (x) e_n (y)=\sin(mx)\sin(nx) .
\]
This leads us to consider the operator $S$, analogous to the one variable case, which maps the basis $e_{m,n}$ to $m^{-s} n^{-t}$.  In the same way we define
\[
T_{\varphi} f(x,y) = \sum_{n} \langle f, e_{m,n} \rangle \, \varphi_{m,n} (x,y).
\]
Our first goal is to extend Theorem \ref{teo1} to double Dirichlet series.  Repeating the arguments used for standard Dirichlet series  we see that the equality
$$
T_{\varphi} = S^{-1} \circ M_{S\varphi}\circ S.
$$
remains valid.  We also need to define the Hardy space of double Dirichlet series
\[
\Ha^2_2 = \left\{ \sum a_{m,n} m^{-s}n^{-t}: \sum_{m,n}|a_{m,n}|^2< \infty  \right\}.
\]
By the Cauchy-Schwarz inequality, just as in the one variable case, every Dirichlet series in $\Ha^2_2$ converges absolutely on $ \C_{\frac{1}{2}}\times \C_{\frac{1}{2}}$.
Moreover, for each $\varepsilon>0$ we set  $$C_{\varepsilon}= \sum_m \frac{1}{m^{1 + 2 \re s}}  \sum_n \frac{1}{m^{1 + 2 \re t}}.$$ Using Cauchy-Schwarz inequality again we get
\[
\sup\limits_{(s,t)  \, \in \, \C_{1/2 +\varepsilon} \times \C_{ 1/2 +\varepsilon} } | D(s,t)| \le C_{\varepsilon} \norm{D}_{\Ha^2_2}.
\]
This means that  $\Ha^2_2$ is continuously contained in $C_{b}
\big(\C_{\frac{1}{2} +\varepsilon} \times \C_{\frac{1}{2} +\varepsilon}\big)$, the space of continuous and bounded functions on $\C_{\frac{1}{2} +\varepsilon} \times \C_{\frac{1}{2} +\varepsilon }$.

Following the previous notation, we also write $\mathcal{M}_2$ for the set of all multipliers of $\Ha^2_2$, which turns out to be a space of double Dirichlet series. Each double Dirichlet series in $\mathcal{M}_2$ defines a bounded multiplication operator $M_D$ on  $\Ha^2_2$.

With  these definitions and following exactly the same proof we are able to extend Theorem \ref{teo1} as follows.

\begin{theorem}\label{teo1-doble} Let $\varphi \in L^2 \left( {I^2}\right)$ and $\varphi_{m,n} (x) = \varphi(mx, ny)$. Then the following assertions hold:
	\begin{itemize}
		\item The system $\{\varphi_{m,n}\}$ is a Bessel sequence if and only if $S\varphi \in \mathcal{M}_2 $.
		\item The system $\{\varphi_{m,n}\}$ satisfies the lower frame bound  if and only if $1/S\varphi \in \mathcal{M}_2 $.

\end{itemize}
\end{theorem}

The above theorem describes some properties of the system $\{\varphi_{m,n}\}$ in terms of the space $\mathcal{M}_2$, so we look for a characterization of this space analogous to Theorem \ref{Multipliers}.
Let us start by analyzing if an analogue to Corollary \ref{3settings} is   valid for multiple Dirichlet series.

 In  \cite{CaGaMa19}, Castillo, García and Maestre proved  that the spaces $\mathcal{H}^{\infty}(\C_{0}^k)$, $k \in \mathbb{N}$, are all isometrically isomorphic independently of their dimension. Indeed, they showed that these spaces are all isometrically isomorphic to the space $H^{\infty}(B_{c_0})$. As a key step they extended Theorem \ref{H inf DirSer} and showed in \cite[Theorem 3.5]{CaGaMa19} that the multiple Bohr transform given by
\begin{align*}
	\mathcal{B}: H^{\infty}(B_{c^{k}_0}) &\xrightarrow{\phantom{ a_{p^{\alpha}}=c_{\alpha}}} \Ha^{\infty} (\C^{k}_{0}) \\
	 \sum_{\alpha_j \in \N_0 ^{(\N)}} c_{\alpha_1, \ldots , \alpha_k} z_1^{\alpha_1}\ldots z_k^{\alpha_k} &\xrightarrow{\phantom{a_{p^{\alpha}}=c_{\alpha}}} \sum_{m_1, \ldots, m_k =1}^{\infty} \frac{a_{\primos^{\alpha_1}, \ldots, \primos^{\alpha_k}}}{m_1^{s_1} , \ldots, m_{k}^{s_k}},
\end{align*}
is an isometric isomorphism. Since the spaces $c^k_0$ and $c_{0}$ are isometrically isomorphic, the aforementioned isometric isomorphism among the spaces $\mathcal{H}^{\infty}(\C_{0}^k)$ is obtained. Moreover, from \cite[Remark 3.6]{CaGaMa19} we see that given $f\in H^{\infty}(B_{c^{k}_0})$ and taking $D=\mathfrak{B}(f)$
 we have $$D(s_1, s_2, \ldots, s_k) = f \left( \frac{1}{\primos^{s_1}},\frac{1}{\primos^{s_2}},\ldots, \frac{1}{\primos^{s_k}} \right)
 $$
 for every $ (s_1, \ldots, s_k) \in \C^k_{0}$.

What is missing is the Fourier analysis point of view for multiple Dirichlet series. Our goal is to develop this third setting for the multiple case. For the sake of clarity,  we consider again  double Dirichlet series. Accordingly, we start by defining
\begin{equation*}
H^{\infty}(\mathbb{T^{\infty}\times \mathbb{T^{\infty}}}) = \{f \in L^{\infty}(\mathbb{T^{\infty}\times \mathbb{T^{\infty}}}) :  \widehat{f}(\gamma , \beta) \neq 0 \textit{ only if }  \beta_j, \alpha_i \geq 0 \textit{ for all } i, j \}.
\end{equation*}
In this direction, we give a version of  Corollary \ref{3settings} for double Dirichlet series, which complements \cite[Theorem 4.33]{CaGaMa19}.

\begin{proposition}\label{TeoHinfC2}
	The following equalities (as Banach spaces) hold
$$ \mathcal{H}^{\infty}(\C_{0}^2)  = H^{\infty}(B_{c^2_0}) = H^{\infty}(\mathbb{T^{\infty}}\times \mathbb{T^{\infty}}). $$
This is provided by the multiple Bohr transform which identifies Dirichlet, monomial and Fourier coefficients. Also,
$$ \mathcal{H}^{2}(\C_{0}^2)  = H^{2}(B_{c^2_0}) = H^{2}(\mathbb{T^{\infty}}\times \mathbb{T^{\infty}}). $$
\end{proposition}

\begin{proof} The first equality is just \cite[Theorem 4.33]{CaGaMa19}
as we have already mentioned, $H^{\infty}(c^2_0) = H^{\infty}(c_0)$  isometrically.  Moreover, the  isometric isomorphism and its inverse are, respectively, the  composition with
\begin{eqnarray*}
\Phi:  B_{c_0} \times B_{c_0}\,  &\longrightarrow & B_{c_0} \\
 \left(\{x_i\}_i,\{y_l\}_l \right) &\mapsto& \{x_1,y_1,x_2,y_2,\dots\}
\end{eqnarray*}
and the inverse
\begin{eqnarray*}
 \Phi^{-1}:  B_{c_0}  &\longrightarrow & B_{c_0}  \times B_{c_0}\\
  z &\mapsto& \left( \{z_{2i}\}_i, \{z_{2l-1}\}_l \right).
\end{eqnarray*}
Analogous mappings can be defined between $\T^{\infty}\times \T^{\infty} $ and $\T^{\infty}$, which give that  $H^{\infty} (\mathbb{T^{\infty}\times \mathbb{T^{\infty}}}) = H^{\infty} (\mathbb{T}^{\infty})$ isometrically.

A combination of these facts with Corollary \ref{3settings} gives
$$ {H}^{\infty}(B_{c_0} \times B_{c_0})= H^{\infty}(B_{c_0})  = H^{\infty}(\mathbb{T^{\infty}}) =  H^{\infty}(\mathbb{T^{\infty}\times \mathbb{T^{\infty}}})   .$$
A careful analysis of the involved mappings (i.e., $\Phi$ defined above, its inverse and their analogues for the torus) shows that this chain of isomorphisms identify monomial coefficients on the left  with Fourier coefficients on the right.

The proof of second assertion is analogous (or even simpler)
\end{proof}

Our next objective is to prove that
\begin{equation}\label{eq-mult2var}
\mathcal{M}(\mathcal{H}^{2} (\C^2_{0})) = \mathcal{H}^{\infty}( \C^2_{0}).
\end{equation}
We have already mentioned that $\mathcal{M}(H^2 (\mathbb{T^{\infty}})) = H^{\infty}(\mathbb{T^{\infty}})$. This is also immediately true for $\mathbb{T}^{\infty} \times \mathbb{T}^{\infty}$, that is,
\begin{equation}\label{MultH}
\mathcal{M}(H^2 (\mathbb{T^{\infty}\times T^{\infty}})) = H^{\infty}(\mathbb{T^{\infty} \times \mathbb{T^{\infty}}}).
\end{equation}
Now, the central idea is to combine \eqref{MultH} with Proposition~\ref{TeoHinfC2} (more precisely, with the equality $ \mathcal{H}^{\infty}(\C_{0}^2) = H^{\infty}(\mathbb{T^{\infty}}\times \mathbb{T^{\infty}})$, and the analogous result for the Hilbert space case: $$ \mathcal{H}^2 (\C_{0}^2) = H^{2}(\mathbb{T}^{\infty} \times \mathbb{T}^{\infty}) .$$  A careful analysis of the isomorphisms together with the multiplicativity of Bohr's transform in the multivariate case gives \eqref{eq-mult2var}.


Now we are ready to present the analogue to Theorem~\ref{teo2} for double Dirichlet series. With all the results of this section in mind, the proof of Theorem \ref{teo22} follows exactly as in the case of ordinary Dirichlet series, so we omit it.

\begin{theorem}\label{teo22} Let $\varphi \in L^{2}(I^2)$, $D= S\varphi$ and $f=\mathcal{B}^{-1}(D)$. The following statements are equivalent:
	\begin{itemize}
		\item[$i)$] The system $\{\varphi_{m,n}\}_{m,n}$ is a Riesz sequence.
		\item[$ii)$] The function $f$ belongs to $H^{\infty}(\mathbb{T^{\infty}\times \mathbb{T^{\infty}}})$ and there exist $\varepsilon >0$ and a set $B \subseteq \mathbb{T^{\infty}}\times \mathbb{T^{\infty}}$ of total measure such that
		\begin{equation}
		|f(z_1, z_2)| \geq \varepsilon, \qquad \textit{for all } (z_1, z_2) \in B \subset \mathbb{T^{\infty}}\times \mathbb{T^{\infty}}.
		\end{equation}
	\end{itemize}
	In this case,
	\begin{equation}
	\varepsilon \left(\sum |b_n|^2  \right) ^{1/2} \le \norm[\Big]{\sum b_n \varphi_n}_{L^2} \le \norm{D}_{\mathcal{H}^{\infty}} \left(\sum |b_n|^2  \right) ^{1/2},
	\end{equation}
	for every finite sequence of scalars.
	
\end{theorem}
In this setting we have not been able to prove an equivalence similar to condition $iii)$ from Theorem \ref{teo2}. The main difficulty arises when we want to define an analogue of $D^{\gamma}$.

\end{document}